\documentclass[12pt]{article}
\usepackage{amsmath}
\usepackage{amssymb}
\usepackage{amsfonts}
\usepackage{amsmath, dsfont}
\usepackage{amsthm}
\usepackage[margin=1in]{geometry}
\usepackage{graphicx}
\usepackage{graphics}
\usepackage{color}

\newtheorem{theorem}{Theorem}
\newtheorem{lemma}{Lemma}
\newtheorem{corollary}{Corollary}
\newtheorem{proposition}{Proposition}

\begin{document}

\title{Tightness for Maxima of Generalized Branching Random Walks}
\author{Ming Fang\thanks{School of Mathematics, University of Minnesota,
206 Church St. SE, Minneapolis, MN 55455, USA. The work was partially
supported by NSF grant DMS-0804133.}}
%\date{October 28, 2010}
\maketitle
{\abstract
We study generalized branching random walks, which allow time dependence and local dependence between siblings. Under appropriate tail assumptions, we prove the tightness of $F_n(\cdot-Med(F_n))$, where $F_n(\cdot)$ is the maxima distribution at time $n$ and $Med(F_n)$ is the median of $F_n(\cdot)$. The main component in the argument is a proof of exponential decay of the right tail $1-F_n(\cdot-Med(F_n))$.
}

\section{Introduction}

We study the maxima of a class of generalized branching random walks (GBRW for short), which are governed by a family of branching rules $\{p_{n,k}\}_{n\geq 0,k\geq 1}$ and displacement laws $\{G_{n,k}\}_{n\geq 0,k\geq 1}$. Specifically, the $p_{n,k}$s are nonnegative reals such that $\sum_{k=1}^{\infty}p_{n,k}=1$ and $\sum_{k=1}^{\infty}kp_{n,k}< \infty$ for each $n\geq 0$, and the $G_{n,k}$s are distribution functions on $\mathds{R}^{k}$ for each $n$ and $k$. The GBRW is defined recursively as follows. At time 0, a particle $o=\overline{1}$ is at location 0. Suppose that $v=\overline{1\alpha_1\dots\alpha_n}$ ($\alpha_i\in\mathds{N}$) is a particle at location $S_v$ at time $n$. At time $n+1$, $v$ dies and gives birth to $K_v\geq 1$ (random) offspring. We denote the offspring of $v$ at generation $n+1$ by $\{\overline{v1},\dots,\overline{vK_v}\}$ and their locations by $\{S_v+X_{v,1},\dots,S_v+X_{v,K_v}\}$, respectively. Let $\mathds{D}$ be the collection of all the particles at any time and $\mathds{D}_n$ the ones alive at time $n$. We consider the case where the random vectors $\{(K_v,X_{v,1},\dots,X_{v,K_v})\}_{v\in\mathds{D}}$ indexed by particles are independent and have distributions
\begin{equation}\label{branching}
  P\left(K_v=k\big|v\in\mathds{D}_n, \mathcal{F}_n\right)=p_{n,k}
\end{equation}
and
\begin{eqnarray}\label{offspring displacement}
  &&P\left(X_{v,1}\leq x_1,\dots,X_{v,K_v}\leq x_{K_v}\big|v\in\mathds{D}_n,K_v=k,\mathcal{F}_n\right)=G_{n,k}(x_1,\dots,x_k) \nonumber\\
  &&\;\;\;\;\;\;\;\;\;\;\;\;\;\;\;\;\;\;\;\;\;\;\;\;\;\;\;\;\;\;\;\;\;\;\;\;\;\;\;
  \;\;\;\;\;\;\;
  \text{for}\;\; n=0,1,\dots \text{and}\;\; k=1,2,\dots,
\end{eqnarray}
where $\mathcal{F}_n=\sigma\{S_u|u\in\mathds{D}_k,k=0,1,\dots,n\}$ is the $\sigma$-field generated by the GBRW by time $n$.

We are concerned with the maximal displacement of particles at time $n$, i.e., $\mathcal{M}_n=\max_{v\in\mathds{D}_n}S_v$. Let $F_n(\cdot)$ be the distribution function of $\mathcal{M}_n$ and set $\bar{F}_n(\cdot)=1-F_n(\cdot)$. Under some assumptions, we want to prove the tightness of the sequence of re-centered distributions $F_n\left(\cdot-Med(F_n)\right)$, where $Med(F_n)$ is the median of $F_n$. See Section \ref{assumption n result} and Section \ref{assumption n result 2} for two different sets of assumptions under which tightness can be proved.

From the previous description, our GBRW allows time dependence (through the n parameter) and some local dependence (reflected by the joint distribution $G_{n,k}$). We will review some of the existing literature and make some comparison. Dekking and Host \cite{DH} (1991) gave a short proof for tightness of $F_n\left(\cdot-Med(F_n)\right)$ when the offspring displacements are uniformly bounded and possibly have some local dependence and time dependence; the technique strongly depends on one side boundedness. Addario-Berry and Reed \cite{AR} (2009) proved that $\mathcal{M}_n-E\mathcal{M}_n$ is exponentially tight when the offspring displacements are i.i.d. and satisfy some large deviation assumptions. By modifying the arguments in \cite{AR}, \cite{Br1} and \cite{DH}, it is possible to extend the tightness result to the case when the offspring displacements are unbounded and have local dependence but not time dependence. See \cite{BZ3} (2010) for using this method to prove the tightness of maxima of modified branching random walks derived from Gaussian free field. In a different direction, Bramson and Zeitouni \cite{BZ} (2009) provided an analytic method to prove tightness of the maximal displacement when the offspring displacements distributions depend on time and satisfy some tail conditions. \cite{BZ} assumed that the offspring displacements are i.i.d. and used a recursion to derive their results. When local dependence comes into play, the recursion, see \eqref{recursion} below, loses some of its nice properties. Therefore, the approach in \cite{BZ}, which is based on the introduction of an appropriate Lyapunov function, does not apply directly here, rather, it needs to be modified to take advantage of a recursion bound, see \eqref{bound} below.

In order to find a recursion, one needs to look at GBRWs starting from particles at some intermediate time. For any integer $m$ and $v=\overline{1\alpha_1\dots\alpha_m}\in \mathds{D}_m$, the process $\{S_u-S_v|u=\overline{1\alpha_1\dots\alpha_m\beta_1\dots\beta_k}\in\mathds{D}_{m+k},
\beta_k\in\mathds{N}, k=1,2,\dots\}$ is a GBRW governed by branching rules $\{p_{n+m,k}\}_{n\geq 0,k\geq 1}$ and displacement laws $\{G_{n+m,k}\}_{n\geq 0,k\geq 1}$. For $n>m$, the maximal displacement at time $n-m$ is denoted by $\mathcal{M}^v_n$. $\{\mathcal{M}^v_n\}_{v\in\mathds{D}_m}$ are i.i.d. random variables whose distribution is denoted by $F^m_n(\cdot)$. Again set $\bar{F}^m_n(\cdot)=1-F^m_n(\cdot)$. Note that $F_n(\cdot)=F^0_n(\cdot)$, $\bar{F}_n(\cdot)=\bar{F}^0_n(\cdot)$ and $\bar{F}_n^n(\cdot)=1_{\{x< 0\}}(\cdot)$.

One obtains a recursion regarding $F^m_n(\cdot)$ by looking at the first generation of GBRWs starting from particles at time $m$. For $n>m$,
\begin{equation*}
  F^m_n(x)=\sum_{k=1}^{\infty}p_{m,k}\int_{\mathds{R}^k}\prod_{i=1}^k F^{m+1}_n(x-y_i)d^kG_{m,k}(y_1,\dots,y_k).
\end{equation*}
Inspired by \cite{BZ}, we consider a recursion for the tail distribution $\bar{F}_n^m(\cdot)$. For $n>m$, the above equation is equivalent to
\begin{equation}\label{recursion}
  \bar{F}_n^m(x)=1-\sum_{k=1}^{\infty}p_{m,k}\int_{\mathds{R}^k}\prod_{i=1}^k \left(1-\bar{F}^{m+1}_n(x-y_i)\right)d^kG_{m,k}(y_1,\dots,y_k).
\end{equation}
Without loss of generality, for any $n,k>0$, we assume $G_{n,k}$ has the same marginal distributions, i.e.,
\begin{equation}\label{equal marginal}
  g_{n,k}(x)=\int_{\mathds{R}^{k-1}}d^{k-1}G_{n,k}(y_1,\dots,y_{i-1},x,y_{i+1},\dots,y_k)
\;\;\text{for any }1\leq i\leq k.
\end{equation}
Otherwise, one can replace $G_{n,k}$ by $\tilde{G}_{n,k}$ defined by
$$\tilde{G}_{n,k}(x_1,\dots,x_k)=\frac{1}{k!}\sum_{\pi\in\mathcal{P}_k}
G_{n,k}(x_{\pi(1)},\dots,x_{\pi(k)})$$
where $\mathcal{P}_k$ denotes all the permutations on $\{1,\dots,k\}$. Then $\tilde{G}_{n,k}$ has the same marginal distributions and one can easily check that recursion \eqref{recursion} is the same for $G_{n,k}$ and $\tilde{G}_{n,k}$.

To apply an approach similar to \cite{BZ}, we introduce two functions
\begin{equation}\label{Q_12}
Q_{1,k}(u)=1-(1-u)^k\;\;\text{and}\;\;Q_{2,k}(u)=ku\;\;\text{for}\;0\leq u\leq 1.
\end{equation}
We will work with the following recursion inequality derived from \eqref{recursion}, instead of \eqref{recursion} itself.
\begin{lemma}\label{lemma bound}
Assume $\bar{F}_n^m(x)$ satisfies the recursion \eqref{recursion}, then the following recursion bounds hold for $n>m$,
  \begin{equation}\label{bound}
  \sum_{k=1}^{\infty}p_{m,k}g_{m,k}*Q_{1,k}(\bar{F}_n^{m+1})(x)\leq \bar{F}_n^m(x)\leq \sum_{k=1}^{\infty}p_{m,k}g_{m,k}*Q_{2,k}(\bar{F}_n^{m+1})(x),
\end{equation}
where $*$ is the convolution defined by $f*g(x)=\int_{-\infty}^{\infty}f(x-y)dg(y)$ for any two functions $f(x)$ and $g(x)$ whenever the integral makes sense.
\end{lemma}
\begin{proof}
\quad We begin by proving the upper bound in \eqref{bound}. Rewrite \eqref{recursion} as
$$\bar{F}_n^m(x)=\sum_{k=1}^{\infty}p_{m,k}\int_{\mathds{R}^k}
  \left(1-\prod_{i=1}^{k}\left(1-\bar{F}_n^{m+1}(x-y_i)\right)\right)d^kG_{m,k}(y_1,\dots,y_k).$$
Using the inequality that $1-\prod_{i=1}^{k}(1-x_i)\leq \sum_{i=1}^kx_i$ for $0\leq x_i\leq 1$ and the fact that $G_{m,k}(\cdot,\dots,\cdot)$ has the same marginal distributions $g_{m,k}(\cdot)$, one obtains that the above quantity is at most
$$\sum_{k=1}^{\infty}p_{m,k}\int_{\mathds{R}^k}\sum_{i=1}^{k}\bar{F}_n^{m+1}(x-y_i)
  d^kG_{m,k}(y_1,\dots,y_k)= \sum_{k=1}^{\infty}p_{m,k}\int_{\mathds{R}}k\bar{F}_n^{m+1}(x-y)dg_{m,k}(y).$$
Together with the definition of $Q_{2,k}$, c.f. \eqref{Q_12}, one obtains the upper bound in \eqref{bound}.

We next prove the lower bound in \eqref{bound}. Applying H\"{o}lder's inequality to \eqref{recursion}, one obtains that
$$\bar{F}_n^m(x)\geq 1-\sum_{k=1}^{\infty}p_{m,k}\prod_{i=1}^{k}\left(\int_{\mathds{R}^k}
  \left(1-\bar{F}_n^{m+1}(x-y_i)\right)^kd^kG_{m,k}(y_1,\dots,y_k)\right)^{1/k}.$$
Again, since $G_{m,k}(\cdot,\dots,\cdot)$ possesses the same marginal distributions $g_{m,k}(\cdot)$, the right side above equals
\begin{eqnarray*}
&&1-\sum_{k=1}^{\infty}p_{m,k}\prod_{i=1}^{k}\left(\int_{\mathds{R}}
  \left(1-\bar{F}_n^{m+1}(x-y)\right)^kdg_{m,k}(y)\right)^{1/k} \\
&=& 1-\sum_{k=1}^{\infty}p_{m,k}\left(\int_{\mathds{R}}
  \left(1-\bar{F}_n^{m+1}(x-y)\right)^kdg_{m,k}(y)\right)\\
&=& \sum_{k=1}^{\infty}p_{m,k}\left(\int_{\mathds{R}}\left(1-
  \left(1-\bar{F}_n^{m+1}(x-y)\right)^k\right)dg_{m,k}(y)\right).
\end{eqnarray*}
Together with the definition of $Q_{1,k}$, see \eqref{Q_12}, one obtains the lower bound in \eqref{bound}.
\end{proof}

\section{Assumptions and statement of result for bounded branching.}\label{assumption n result}
In this section, we discuss the tightness property in the case where the offspring number is uniformly bounded. To state our result, we need some assumptions both on the branching and displacement laws. We introduce assumptions concerning the branching mechanism.
\begin{itemize}
  \item[(B1)] $\{p_{n,k}\}_{n\geq 0}$ possess a uniformly bounded support, i.e., there exists an integer $k_0>1$ such that $p_{n,k}=0$ for all $n$ and $k\notin\{1,\dots,k_0\}$.
  \item[(B2)] The mean offspring number is uniformly greater than 1 by some fixed constant. I.e., there exists a real number $m_0>1$ such that $\inf_{n}\{\sum_{k=1}^{k_0}kp_{n,k}\}>m_0$.
\end{itemize}
We introduce the following assumptions on the displacement laws $G_{n,k}$ for those $n$s and $k$s such that $p_{n,k}\neq 0$.
\begin{itemize}
  \item[(MT1)] For some fixed $\epsilon_0< \frac{1}{4}\log m_0\wedge 1$, there exists an $x_0$ such that $\bar{g}_{n,k}(x_0)\geq 1-\epsilon_0$ for all $n$ and $k$, where $\bar{g}_{n,k}(x)=1-g_{n,k}(x)$. By shifting, we may and will assume that $x_0=0$, that is, $\bar{g}_{n,k}(0)\geq 1-\epsilon_0$.
  \item[(MT2)] There exist $a>0$ and $M_0>0$ such that $\bar{g}_{n,k}(x+M)\leq e^{-aM}\bar{g}_{n,k}(x)$ for all $n,k$ and $M>M_0$, $x\geq 0$.
  \item[(GT)] For any $\eta_1>0$, there exists a $B>0$ such that $G_{n,k}(B,\dots,B)\geq 1-\eta_1$ and $G_{n,k}([-B,\infty)^k)$ $\geq 1-\eta_1$ for all $n$ and $k$. (With an abuse of notation, $G_{n,k}$ is also used here as a function on measurable sets defined by $G_{n,k}(A):=\int_Ad^kG_{n,k}(x_1,\dots,x_k)$ for $A\subset \mathds{R}^k$. See \eqref{offspring displacement} for the definition of $G_{n,k}$ as a distribution function on $\mathds{R}^k$.)
\end{itemize}
Assumptions (MT1) and (MT2) concern the marginal distributions. (MT1) prevents too much mass drifting to $-\infty$, while (MT2) guarantees that the right tails of the marginals decay at least exponentially. (GT) concerns the joint distribution of the movements and prevents any step from being too negative or too positive to dominate the walk. Now we are ready to state our main theorem.
\begin{theorem}\label{tightness}
  Under the above assumptions (B1), (B2), (MT1), (MT2) and (GT), the family of the recentered maxima distributions $\{F_n\left(\cdot-Med(F_n)\right)\}_{n\geq 0}$ is tight.
\end{theorem}

Theorem \ref{tightness} is proved in Section \ref{sec_proof_theorem}, with the proofs of some propositions deferred to Section \ref{sec_proof_proposition}. With an analysis of a Lyapunov function, we control the right tails of distributions $F_n\left(\cdot-Med(F_n)\right)$. Then we use assumption (GT) together with the right tail property to control the behavior of left tails of the distributions. Using a similar approach, we can also prove a variation of Theorem \ref{tightness} under slightly different assumptions in Section \ref{assumption n result 2}.

\section{A Lyapunov function, main induction and proof of Theorem \ref{tightness}}\label{sec_proof_theorem}

This section follows \cite{BZ}, with some minor revisions, in introducing a Lyapunov function. Namely, for a choice of $\epsilon_1$, $b$ and $M$ (to be determined later), we define the Lyapunov function $L(\cdot)$ as
\begin{equation}\label{Lfunction}
  L(u)=\sup_{\{x:u(x)\in (0,\frac{1}{2}]\}} l(u;x),
\end{equation}
where
\begin{equation}\label{lfunction}
  l(u;x)=\log\left(\frac{1}{u(x)}\right)+\log_b\left(1+\epsilon_1-\frac{u(x-M)}{u(x)}\right)_+\;.
\end{equation}
Here $(x)_+=x\vee 0$, and we take the convention that $\log 0=-\infty$.

As in \cite{BZ}, the heart of the proof is contained in the following proposition.
\begin{proposition}\label{Lyapunov}
  Under assumptions (B1), (B2), (MT1) and (MT2), there is a choice of $\epsilon_1$, $b$ and $M$ such that $\sup_{m\leq n}L(\bar{F}_n^m)<C$ for some finite number $C>0$.
\end{proposition}
The proof of Proposition \ref{Lyapunov} will take the bulk of the paper, and is detailed in Section \ref{sec_proof_proposition}. Before proving it, we discuss its consequences. As in \cite[Corollary 2.8]{BZ}, the same proof, using Proposition \ref{Lyapunov}, yields the following
\begin{corollary}\label{RT}
  Let the assumptions (B1), (B2), (MT1) and (MT2) hold. Then, there exists $\delta_1$ such that, for all $n$ and $m\leq n$,
  \begin{equation}\label{RTinequality}
    \bar{F}_n^m(x)\leq \delta_1\;\; \text{implies}\;\; \bar{F}_n^m(x-M)\geq (1+\frac{\epsilon_1}{2})\bar{F}_n^m(x).
  \end{equation}
\end{corollary}
This corollary gives a desired control over the behavior of the right tail of $\bar{F}_n^m(\cdot)$. We next control the left tail. First, one obtains the following pointwise bounds for the integral \eqref{recursion}.
\begin{lemma}\label{Assumption in BZ}
  The assumption (GT) implies that, for any $\eta_1>0$, there exists a $B$ such that  \begin{equation}\label{pwbounds}
    Q_m(\bar{F}_n^{m+1})(x+B)-\eta_1\leq \bar{F}_n^m(x)\leq Q_m(\bar{F}_n^{m+1})(x-B)+\eta_1,
  \end{equation}
  where $Q_m(u)=\sum_{k=1}^{\infty}p_{m,k}\left(1-(1-u)^k\right)$.
\end{lemma}
\begin{proof}
\quad For any $\eta_1>0$, choose the $B$ as in the assumption (GT). The upper bound is obtained by only considering the integral over $(-\infty,B]^k$ in \eqref{recursion}.
$$\bar{F}_n^m(x)\leq 1-\sum_{k=1}^{\infty}p_{m,k}\int_{(-\infty,B]^k}\prod_{i=1}^k \left(1-\bar{F}^{m+1}_n(x-y_i)\right)d^kG_{m,k}(y_1,\dots,y_k).$$
By the monotonicity of $\bar{F}_n^m(\cdot)$, the right side is less than
$$1-\sum_{k=1}^{\infty}p_{m,k}\left(1-\bar{F}_n^{m+1}(x-B)\right)^kG_{m,k}(B,\cdots,B).$$
For any $\eta_1$, choose $B$ as in assumption (GT). Then $G_{m,k}(B,\cdots,B)\geq 1-\eta_1$, and the above quantity is less than or equal to
\begin{eqnarray*}
  &&1-\sum_{k=1}^{\infty}p_{m,k}\left(1-\bar{F}_n^{m+1}(x-B)\right)^k(1-\eta_1)\\
  &=& Q_m(\bar{F}_n^{m+1})(x-B)+\eta_1 \sum_{k=1}^{\infty}p_{m,k}\left(1-\bar{F}_n^{m+1}(x-B)\right)^k\\
  &\leq & Q_m(\bar{F}_n^{m+1})(x-B)+\eta_1,
\end{eqnarray*}
proving the upper bound in \eqref{pwbounds}. To obtain the lower bound, first rewrite \eqref{recursion} as
$$\bar{F}_n^m(x)=\sum_{k=1}^{\infty}p_{m,k}\int_{\mathds{R}^k}\left(1-\prod_{i=1}^k \left(1-\bar{F}^{m+1}_n(x-y_i)\right)\right)d^kG_{m,k}(y_1,\dots,y_k).$$
By restricting the above integral to $[-B,\infty)^k$, one has a lower bound on $\bar{F}_n^m$,
$$\bar{F}_n^m(x)\geq \sum_{k=1}^{\infty}p_{m,k}\int_{[-B,\infty)^k}\left(1-\prod_{i=1}^k \left(1-\bar{F}^{m+1}_n(x-y_i)\right)\right)d^kG_{m,k}(y_1,\dots,y_k).$$
Since $\bar{F}_n^{m+1}(x)$ is decreasing in $x$ and $G_{m,k}\left([-B,\infty)^k\right)\geq 1-\eta_1$ as in assumption (GT), one has
\begin{eqnarray*}
  \bar{F}_n^m(x)&\geq & \sum_{k=1}^{\infty}p_{m,k}\left(1-\left(1-\bar{F}_n^{m+1}(x+B)\right)^k\right)
  G_{m,k}\left([-B,\infty)^k\right)\\
  &\geq & \sum_{k=1}^{\infty}p_{m,k}\left(1-\left(1-\bar{F}_n^{m+1}(x+B)\right)^k\right)
  (1-\eta_1)\\
  &=& Q_m(\bar{F}_n^{m+1})(x+B)-\eta_1
  \sum_{k=1}^{\infty}p_{m,k}\left(1-\left(1-\bar{F}_n^{m+1}(x+B)\right)^k\right)\\
  &\geq & Q_m(\bar{F}_n^{m+1})(x+B)-\eta_1,
\end{eqnarray*}
proving the lower bound in \eqref{pwbounds} and completing the proof of Lemma \ref{Assumption in BZ}.
\end{proof}

Lemma \ref{Assumption in BZ} almost checks \cite[ASSUMPTION 2.4]{BZ}, except that $Q_m$ depends on $m$. However, with the assumption (B1), $Q_m$ satisfies \cite[T1 and T2 in DEFINITION 2.3]{BZ} uniformly in $m$. Namely, the family of strictly increasing functions $Q_m:[0,1]\to[0,1]$, with $Q_m(0)=0$ and $Q_m(1)=1$, satisfies the following.
\begin{itemize}
  \item[(T1')] $Q_m(x)>x$ for all $x\in(0,1)$. For any $\delta>0$, one can choose $c_{\delta}=1+\frac{m_0-1}{k_0}\delta>1$ such that $Q_m(x)>c_{\delta}x$ for all $x\leq 1-\delta$ and all $m$.
  \item[(T2')] For each $\delta\in (0,1)$, there exists a nonnegative function $g_{\delta}(\epsilon)\to 0$ as $\epsilon\to 0$ (for example, choose $g_{\delta}(\epsilon)=\frac{(1-(1-\delta)^{k_0})}{k_0\delta}
      \left(\frac{1+\epsilon}{\delta+\epsilon}\right)^{k_0-1}\epsilon$) such that, for any $m$, if $x\geq \delta$ and $Q_m\left(\left(1+g_{\delta}(\epsilon)\right)x\right)\leq \frac{1-\delta}{1+\epsilon}$, then $Q_m\left(\left(1+g_{\delta}(\epsilon)\right)x\right)\geq (1+\epsilon)Q_m(x)$.
\end{itemize}
To check the above two properties, one uses the strict convexity of $1-(1-x)^k$ and its monotonicity in $k$. Details are omitted here. From the above (T1') and (T2'), one can deduce the following lemma in exactly the same way as in \cite[Lemma 2.10]{BZ}.
\begin{lemma}\label{LnR}
 Suppose that \eqref{RTinequality} holds for all $m\leq n$ under some choice of $\delta_1,M,\epsilon_1>0$. Also, suppose that assumption (B1) and \eqref{pwbounds} hold. For fixed $\eta_0\in(0,1)$, there exist a constant $\gamma=\gamma(\eta_0)<1$ and a continuous function $f(t)=f_{\eta_0}(t):[0,1]\to [0,1]$, with $f(t)\to_{t\to 0}0$, such that for any $\epsilon\in(0,\frac{1-\eta_0}{\eta_0})$, $\eta\in[\delta_1,\eta_0]$ and large enough $N_1=N_1(\epsilon)$, the following holds. If $M'>M$ and, for any $m<n$, $\bar{F}_n^m(x)\geq \delta_1$,
 $$\bar{F}_n^m(x-M')\leq (1+\epsilon)\bar{F}_n^m(x)\;\;\mbox{and}\;\; \bar{F}_n^m(x-M')\leq \eta,$$
 then
 $$\bar{F}_n^{m+1}(x+N_1-M')\leq \left(1+f(\epsilon)\right)\bar{F}_n^{m+1}(x-N_1)$$
 and
 $$\bar{F}_n^{m+1}(x+N_1-M')\leq \gamma\eta.$$
\end{lemma}

By iterating, the above lemma gives a connection between the left and right tail behavior. That is, by applying Corollary \ref{RT} and Lemma \ref{LnR} several times as in \cite[Proof of Proposition 2.9]{BZ}, the same contrapositive argument proves: for fixed $\eta_0\in(0,1)$, there exist an $\hat{\epsilon}_0=\hat{\epsilon}_0(\eta_0)>0$, an $n_0$ and an $\hat{M}$ such that, if $n>n_0$ and $\bar{F}_n^0(x-\hat{M})\leq \eta_0$, then $\bar{F}_n^0(x-\hat{M})\geq (1+\hat{\epsilon}_0)\bar{F}_n^0(x)$. This will yield the following tightness proposition by recalling that $F_n(\cdot)=F_n^0(\cdot)$.
\begin{proposition}\label{tightprop}
   Suppose that \eqref{RTinequality} holds for all $m\leq n$ under some choice of $\delta_1,M,\epsilon_1>0$. Also, suppose that assumption (B1) and \eqref{pwbounds} hold. Then, the family of recentered maxima distributions $\{F_n\left(\cdot-Med(F_n)\right)\}_{n\geq 0}$ is tight.
\end{proposition}

By now, we have proved Theorem \ref{tightness} assuming that Proposition \ref{Lyapunov} is true. Thus, remained to show is Proposition \ref{Lyapunov}, which we do in the next section.

\section{Analysis of Lyapunov function and proof of Proposition \ref{Lyapunov}}\label{sec_proof_proposition}

In this section we focus on proving Proposition \ref{Lyapunov}, which is an analog of \cite[Theorem 2.7]{BZ}. The same idea works here: the exponential decay of $g_{n,k}$ will not bring much mass from faraway during the recursion. However, the exact approach does not quite apply here. \cite{BZ} deals with the nonlinearity and convolution in a recursion equality separately. In our case, the recursion \eqref{recursion} does not possess such a nice form. Fortunately, we have the recursion inequalities \eqref{bound}. These bounds require one to analyze the nonlinearity and convolution together. Throughout this section, all the sums about $k$ are from $1$ to $k_0$ since assumption (B1) is assumed. We begin with some properties of the two functions in \eqref{Q_12}. $Q_{2,k}(u)=ku$ is simple, and the following simple facts about $Q_{1,k}(u)$ will be used later on.
\begin{lemma}\label{Q_1_LB12}
  There exists a $c_1=c_1(k_0)\geq 1$ such that, for all $1\leq k\leq k_0$ and $0\leq u\leq 1$,
  \begin{equation}\label{Q_1_LB1}
    Q_{1,k}(u)\geq u
  \end{equation}
  and
  \begin{equation}\label{Q_1_LB2}
    ku-c_1u^2\leq Q_{1,k}(u)\leq ku=Q_{2,k}(u).
  \end{equation}
\end{lemma}
Next, we state a choice of $\epsilon_1$, $b$ and $M$ in the Lyapunov function under which Proposition \ref{Lyapunov} is true. Throughout, we fix $k_0,m_0,\epsilon_0, M_0$ and $a$ as in assumptions (B1), (B2), (MT1) and (MT2). Next, we choose $\epsilon_1<\frac{1}{100}$ small, $b>1$ close to 1, $M>100$ big and an auxiliary variable $\kappa<\frac{1}{100}$ small (used later to control the flatness change) such that the following restrictions hold.
\begin{equation}\label{M1}
  M>4M_0\;\; \mbox{and}\;\; e^{-aM/2}\leq (4k_0)^4e^{-aM/2}\leq \frac{1}{100};
\end{equation}
\begin{equation}\label{b1}
\frac{8(2k_0)^{5/2}\epsilon_1^{1/2\log b-3/2}}{(1-\epsilon_0)\kappa^{3/2}}<\frac{1}{2c_1};
\end{equation}
\begin{equation}\label{b2}
  c_1\frac{1+\epsilon_1}{1-\epsilon_0}\epsilon_1^{1/\log b}\leq \sum_{k=1}^{k_0}kp_{n,k}-m_0\;\;\mbox{for all } n;
\end{equation}
\begin{equation}\label{ep_1nkappa}
  \frac{\log m_0}{2}\geq 2(\epsilon_1+\epsilon_0)+\frac{6\kappa}{\log b};
\end{equation}
\begin{equation}\label{M2}
  \frac{aM}{16\log b}\geq 2\left(\epsilon_1+\epsilon_0+\log(4k_0)\right)-\frac{\log\kappa}{\log b};
\end{equation}
\begin{equation}\label{b4}
  \frac{a}{16\log b}\geq \frac{2\log(4k_0)}{M}.
\end{equation}
The above conditions are compatible. In fact, thinking of $\kappa$ as $\beta \log b$, one can choose $\epsilon_1$ and $\beta$ small enough so that \eqref{ep_1nkappa} holds due to the choice of $\epsilon_0$ in assumption (MT1), then one chooses a $b$ close enough to 1 so that \eqref{b1}, \eqref{b2} and \eqref{b4} hold due to the choice of $m_0$ as in assumption (B2), and finally one chooses $M$ large enough so that \eqref{M1} and \eqref{M2} hold.

With the choice of the above $\epsilon_1$, $b$, $M$ and $\kappa$, we can now prove Proposition \ref{Lyapunov}.
\begin{proof}[Proof of Proposition \ref{Lyapunov}]
\quad Choose $C=-\log\frac{1}{2}$. The conclusion $\sup_{m\leq n}L(\bar{F}_n^m)\leq C$ will follow from the claim below
\begin{equation}\label{fact}
L(\bar{F}_n^m)> C\;\; \mbox{implies that}\;\; L(\bar{F}_n^{m+1})>C\;\; \mbox{for any}\;\; m<n.
\end{equation}
Suppose the conclusion is violated, then $L(\bar{F}_n^m)>C$ for some $m\leq n$. Iterating the claim $n-m$ times, one gets $L(\bar{F}_n^n)>C$. However, $L(\bar{F}_n^n)=-\infty$ because $\bar{F}_n^n(x)=1_{\{x< 0\}}(x)$. This contradiction proves proposition \ref{Lyapunov}, assuming claim \eqref{fact}.
\end{proof}

The claim \eqref{fact} follows from the following proposition because of \eqref{bound}.
\begin{proposition}\label{prop_fact}
Suppose
that two non-increasing cadlag functions $u,v:\mathds{R}\to [0,1]$ satisfy
\begin{equation}\label{bound'}
\sum_{k=1}^{k_0}p_kg_k*Q_{1,k}(u)(x)\leq v(x)\leq \sum_{k=1}^{k_0}p_kg_k*Q_{2,k}(u)(x),
\end{equation}
where $p_k$ and $g_k$ satisfy the assumptions in Section \ref{assumption n result} as $p_{n,k}$ and $g_{n,k}$, and $Q_{1,k}$ and $Q_{2,k}$ satisfy Lemma \ref{Q_1_LB12}. Then
\begin{equation}\label{fact'}
 L(v)> C\;\; \mbox{implies that}\;\; L(u)>C.
\end{equation}
\end{proposition}
In order to prove Proposition \ref{prop_fact}, a few observations, notation and lemmas are needed. Starting from $L(v)> C$, one obtains, by definition \eqref{Lfunction} of the Laypunov function, that there exists an $x_1\in\mathds{R}$ such that
\begin{equation}\label{x_1}
v(x_1)\leq \frac{1}{2}\;\; \mbox{and}\;\; l(v;x_1)\geq \max\{C,L(v)-\frac{1}{4}\log m_0\}.
\end{equation}
By definition \eqref{lfunction} of $l(v;x)$, one obtains that $v$ is small and flat at $x_1$ in the following sense:
\begin{equation}\label{epsilon}
1+\epsilon:=\frac{v(x_2)}{v(x_1)}< 1+\epsilon_1
\end{equation}
and
\begin{equation}\label{f_0}
f_0:=v(x_1)<(\epsilon_1-\epsilon)^{1/\log b}e^{-C}< \frac{1}{2},
\end{equation}
where $x_2:=x_1-M$. Using the bounds \eqref{bound'} and \eqref{epsilon}, one gets that
\begin{equation}\label{Q_1Q_2}
 \sum_{k=1}^{k_0}p_kg_k*Q_{1,k}(u)(x_2)\leq(1+\epsilon)
 \sum_{k=1}^{k_0}p_kg_k*Q_{2,k}(u)(x_1),
\end{equation}
from which we will search for a flat piece in $u(x)$ where $u(x)$ is also small.

To control the value of $u(x)$, we derive here some preliminary estimates of $u(x)$ at $x_1$ and $x_2$, which will be used later to control the value of $u(x)$ at other places. For $i=1,2$, first applying the Chebyshev inequality and then applying \eqref{bound'} and the fact $\bar{g}_k(0)\geq 1-\epsilon_0$ from assumption (MT1), one gets
% and the bound \eqref{Q_1_LB1} for $Q_{1,k}$, and one gets
\begin{eqnarray}\label{u(x_i)}
  \sum_{k=1}^{k_0}p_k Q_{1,k}(u)(x_i)&\leq& \sum_{k=1}^{k_0}p_k\frac{1}{\bar{g}_k(0)}\int_{\mathds{R}}
  Q_{1,k}(u)(x_i-y)dg_k(y)\nonumber\\
  &\leq & \frac{1}{1-\epsilon_0}v(x_i).
\end{eqnarray}
This, together with the lower bound \eqref{Q_1_LB1} on $Q_{1,k}$, the definition \eqref{f_0} of $f_0$ and the definition \eqref{epsilon} of $\epsilon$, implies that
\begin{equation}\label{u(x_1)}
  u(x_1)\leq \frac{f_0}{1-\epsilon_0},
\end{equation}
and
\begin{equation}\label{u(x_2)1}
  u(x_2)\leq \frac{1+\epsilon}{1-\epsilon_0}f_0.
\end{equation}
A finer estimate of $u(x_2)$ can be obtained and will be needed. First, using \eqref{u(x_i)} and the lower bound \eqref{Q_1_LB2} on $Q_{1,k}$, one gets
$$\left(\sum_{k=1}^{k_0}kp_k -cu(x_2)\right)u(x_2)\leq \frac{1+\epsilon}{1-\epsilon_0}f_0.$$
By combining the first estimate \eqref{u(x_2)1} of $u(x_2)$, the bound \eqref{f_0} on $f_0$ and the restriction \eqref{b2}, the coefficient multiplying $u(x_2)$ on the left side of the last inequality is at least
\begin{eqnarray*}
&&\sum_{k=1}^{k_0}kp_k-c_1u(x_2)\geq \sum_{k=1}^{k_0}kp_k-c_1\frac{1+\epsilon}{1-\epsilon_0}f_0 \\
&\geq & \sum_{k=1}^{k_0}kp_k-c_1\frac{1+\epsilon_1}{1-\epsilon_0}(\epsilon_1-\epsilon)^{1/\log b}e^{-C}\\
&\geq & \sum_{k=1}^{k_0}kp_k-c_1\frac{1+\epsilon_1}{1-\epsilon_0}\epsilon_1^{1/\log b}\geq m_0.
\end{eqnarray*}
Therefore, we conclude that
\begin{equation}\label{u(x_2)2}
  u(x_2)\leq \frac{1+\epsilon}{m_0(1-\epsilon_0)}f_0
  =\frac{1+\epsilon}{m_0(1-\epsilon_0)}v(x_1).
\end{equation}

To control the flatness of $u(x)$, we define some more auxiliary variables and then state some lemmas. The constants $\delta=\kappa(\epsilon_1-\epsilon)$, $\epsilon'=\epsilon+\delta$, $\epsilon''=\epsilon+2\delta$ and $\epsilon^{(3)}=\epsilon+3\delta$ are defined to monitor the flatness change. Note that $\epsilon,\epsilon',\epsilon'',\epsilon^{(3)}<\epsilon_1$ because $\kappa<\frac{1}{100}$. We somewhat simplify the argument in \cite{BZ}. Set
\begin{equation}\label{y_0}
  y_0=\frac{1}{a}\log \frac{2k_0}{\delta f_0},
\end{equation}
\begin{equation}\label{q}
  q=\inf\{y\geq M/2:u(x_2-y)>(4k_0)^2u(x_1-y)\}
\end{equation}
and
\begin{equation}\label{r}
  r=y_0\wedge\left\{\begin{array}{ll}
    q,&\text{if}\; u(x_2-q)^-\geq (4k_0)u(x_1-(q+\frac{M}{2}));\\
    q-\frac{M}{2},&\text{otherwise,}
  \end{array}\right.
\end{equation}
where $u(x)^-:=\lim_{y\to x-}f(y)$ is the left limit of $f$ at $x$. Intuitively, $q$ is used to denote the first nonflatness place to the left of $x_1$. When $r<y_0$, $r$ is used to denote a nonflat interval, namely, it is easy to check that
\begin{equation}\label{r_steep}
u(x_2-y)\geq (4k_0)u(x_1-y)\;\;\mbox{for all}\;\;y\in(r,r+M/2].
\end{equation}
We can now state the following sequence of lemmas, whose proofs will be discussed in the next subsection. The convention of
$$\int_a^bf(x)dg(x)=\int_{(a,b]}f(x)dg(x)$$
for $a,b\in\mathds{R}$ will be made throughout the rest of the paper.
\begin{lemma}\label{truncation_flat}
Assume that \eqref{f_0} and \eqref{Q_1Q_2} hold. Then,
\begin{equation}\label{truncation}
  \sum_{k=1}^{k_0}p_k\int_{-\infty}^{r}Q_{1,k}(u)(x_2-y)dg_k(y)
  \leq(1+\epsilon')
 \sum_{k=1}^{k_0}p_k\int_{-\infty}^{r}Q_{2,k}(u)(x_1-y)dg_k(y).
\end{equation}
\end{lemma}
\begin{lemma}\label{conv_flat}
  If \eqref{f_0} and \eqref{truncation} are satisfied, then there exist some $1\leq k\leq k_0$ and $r'$ such that
  \begin{equation}\label{linearization}
    \int_{-\infty}^{r'}u(x_2-y)dg_k(y)\leq (1+\epsilon'')\int_{-\infty}^{r'}u(x_1-y)dg_k(y),
  \end{equation}
  where $r'=r$ when $r'>M$.
\end{lemma}
\begin{lemma}\label{flatness}
 Suppose \eqref{linearization} holds. Then either
 \begin{itemize}
 \item[(a)] $u(x_2-y_1)\leq (1+\epsilon^{(3)})u(x_1-y_1)$ for some $y_1\leq r'\wedge M$, or
 \item[(b)] $u(x_2-y_1)\leq (1+\epsilon''-\delta e^{ay_1/8})u(x_1-y_1)$ for some $y_1\in (M,r]$.
 \end{itemize}
\end{lemma}
Lemma \ref{conv_flat} and Lemma \ref{flatness} are analogs of \cite[Lemma 3.5, Proposition 3.2]{BZ}, respectively. Equipped with lemma \ref{flatness}, we are ready to prove Proposition \ref{prop_fact}.

\begin{proof}[Proof of Proposition \ref{prop_fact} assuming Lemma \ref{flatness}]\quad We will compare $L(u)$ and $L(v)$ based on \eqref{u(x_2)2} and Lemma \ref{flatness}. As Lemma \ref{flatness} suggests, two different cases will be discussed separately.

\emph{Case (a)}: Assume $u(x_2-y_1)\leq (1+\epsilon^{(3)})u(x_1-y_1)$ for some $y_1\leq r'\wedge M$. Then, \eqref{u(x_2)2} implies that
$$u(x_1-y_1)\leq u(x_2)\leq \frac{1+\epsilon}{m_0(1-\epsilon_0)}v(x_1).$$
Therefore, it follows by the definition \eqref{lfunction} of $l(u;x)$ that
\begin{eqnarray*}
  l(u,x_1-y_1)-l(v,x_1)
  &\geq &\log\frac{v(x_1)}{u(x_1-y_1)}+ \log_b\frac{\epsilon_1-\epsilon^{(3)}}{\epsilon_1-\epsilon}\\
  &\geq & \log \frac{m_0(1-\epsilon_0)}{1+\epsilon} +\log_b(1-3\kappa)\\
  &\geq & \log m_0 -2(\epsilon_1+\epsilon_0)-\frac{6\kappa}{\log b}\geq \frac{\log m_0}{2},
\end{eqnarray*}
where \eqref{ep_1nkappa} guarantees the last inequality.

\emph{Case (b)}: Assume $u(x_2-y_1)\leq (1+\epsilon''-\delta e^{ay_1/8})u(x_1-y_1)$ for some $y_1\in (M,r]$. Then, the definition \eqref{r} of $r$ and \eqref{u(x_2)2} imply that
$$u(x_1-y_1)\leq (4k_0)^{2y_1/M+2}u(x_1-M/2)\leq (4k_0)^{2y_1/M+2}\frac{1+\epsilon}{m_0(1-\epsilon_0)}v(x_1).$$
Therefore, it follows that
\begin{eqnarray*}
  l(u,x_1-y_1)-l(v,x_1)
  &=&\log\frac{v(x_1)}{u(x_1-y_1)}+ \log_b\frac{\epsilon_1-\epsilon''+\delta e^{ay_1/8}}{\epsilon_1-\epsilon}\\
  &\geq & \log \frac{m_0(1-\epsilon_0)}{(1+\epsilon)(4k_0)^{2y_1/M+2}}+ \log_b(1-2\kappa+\kappa e^{ay_1/8})\\
  &\geq & \log m_0 -2(\epsilon_0+\epsilon_1)-\frac{2\log (4k_0)}{M}y_1 -2\log(4k_0) +\frac{\log \kappa + ay_1/8}{\log b}.
\end{eqnarray*}
Rewrite the last term $\frac{ay_1}{8\log b}$ as $\frac{ay_1}{16\log b}+\frac{ay_1}{16\log b}$, use $y_1\geq M$ in one summand and deduce that the above quantity is at least
$$  \log m_0 -2(\epsilon_0+\epsilon_1+\log(4k_0))+\frac{\log \kappa}{\log b}+\frac{aM}{16\log b}+y_1(\frac{a}{16\log b}-\frac{2\log(4k_0)}{M})
  \geq  \frac{1}{2}\log m_0,$$
where $\eqref{b4}$ and $\eqref{M2}$ guarantee the last inequality.

To wrap the argument up, both cases imply, by \eqref{lfunction}, \eqref{x_1} and $C=-\log\frac{1}{2}$,
$$\log\frac{1}{u(x_1-y_1)}\geq l(u,x_1-y_1) \geq C +\frac{1}{2}\log m_0\geq -\log\frac{1}{2},$$
which implies that $u(x_1-y_1)\leq \frac{1}{2}$. Therefore, by the definition \eqref{Lfunction} of $L(u)$ and \eqref{x_1} again,
$$L(u)\geq l(u,x_1-y_1)\geq l(v,x_1)+\frac{1}{2}\log m_0 \geq L(v)+\frac{1}{4}\log m_0\geq  L(v), $$
from which \eqref{fact'} follows. Thus, the proof of Proposition \ref{prop_fact} is complete.
\end{proof}

\subsection{Proof of lemmas}\label{proof of lemmas}

With the assumption (MT2), the proof of \cite[Proposition 3.2]{BZ} carries over (with some change of notation) to the proof of Lemma \ref{flatness} assuming Lemma \ref{conv_flat}. For completeness, we bring the proof in the appendix. The proof of Lemma \ref{conv_flat} will be presented first, and then the proof of Lemma \ref{truncation_flat}.

\begin{proof}[Proof of Lemma \ref{conv_flat}]
\quad When $q>M/2$, we have $u(x_2-y)\leq (4k_0)^2u(x_1-y)$ for $y\in [M/2,q]$. Thus, one obtains that, for any $y\leq r\leq q$,
$$u(x_2-y)\leq u(x_2-r)\leq (4k_0)^{2r/M+2}u(x_2).$$
Since $r\leq y_0=\frac{1}{a}\log \frac{2k_0}{\delta f_0}$, one has, using \eqref{u(x_2)1}, that the above is at most
$$(4k_0)^{2y_0/M+2}\frac{1+\epsilon}{1-\epsilon_0}f_0
<\frac{2(4k_0)^2}{1-\epsilon_0}(4k_0)^{\frac{2}{aM}\log \frac{2k_0}{\delta f_0}}f_0=\frac{2(4k_0)^2}{1-\epsilon_0}(\frac{2k_0}{\delta f_0})^{\frac{2}{aM}\log(4k_0)}f_0.$$
Note that $\frac{2}{aM}\log (4k_0)<\frac{1}{2}$ from \eqref{M1}. Applying the bound \eqref{f_0} on $f_0$, the above quantity is at most
$$\frac{2(4k_0)^2}{1-\epsilon_0}\frac{\sqrt{2k_0}f_0^{1/2}}{\delta^{1/2}}
=\frac{8(2k_0)^{5/2}f_0^{1/2}}{(1-\epsilon_0)\delta^{3/2}}\delta
<\frac{8(2k_0)^{5/2}(\epsilon_1-\epsilon)^{1/2\log b-3/2}}{(1-\epsilon_0)\kappa^{3/2}}\delta.$$
Therefore, it follows from \eqref{b1} that
\begin{equation}\label{u(x_2-y)small}
 u(x_2-y)\leq \frac{1}{2c_1}\delta \;\;\mbox{for any}\;\;y\leq r.
\end{equation}
This, combined with \eqref{Q_1_LB2}, implies that, for any $1\leq k\leq k_0$ and $y\leq r_1$,
\begin{eqnarray*}
  &&Q_{1,k}(u)(x_2-y)\geq ku(x_2-y)-c_1\left(u(x_2-y)\right)^2\\
  &=& ku(x_2-y)(1-\frac{c_1}{k}u(x_2-y))\geq
  ku(x_2-y)(1-\frac{1}{2}\delta).
\end{eqnarray*}
Applying the above bound and the definition \eqref{Q_12} of $Q_{2,k}(u)$ in the first inequality, and \eqref{truncation} in the second, one has
\begin{eqnarray}\label{contradiction}
  &&\frac{\sum_{k=1}^{k_0}kp_k\int_{-\infty}^{r}u(x_2-y)dg_k(y)} {\sum_{k=1}^{k_0}kp_k\int_{-\infty}^{r}u(x_1-y)dg_k(y)}\nonumber\\
  &\leq &\frac{1}{1-\frac{1}{2}\delta} \frac{\sum_{k=1}^{k_0}p_k\int_{-\infty}^{r}Q_{1,k}(u)(x_2-y)dG_k(y)} {\sum_{k=1}^{k_0}p_k\int_{-\infty}^{r}Q_{2,k}(u)(x_1-y)dG_k(y)}\\
  &\leq & \frac{1+\epsilon'}{1-\frac{1}{2}\delta}\leq 1+\epsilon''.\nonumber
\end{eqnarray}
If the conclusion of the lemma does not hold, i.e., for all $1\leq k\leq k_0$,
$$ \int_{-\infty}^{r}u(x_2-y)dg_k(y)> (1+\epsilon'')\int_{-\infty}^{r}u(x_1-y)dg_k(y),$$
one obtains a contradiction to \eqref{contradiction}. This completes the proof of Lemma \ref{conv_flat} in case $q>M/2$.

When $q=M/2$ and $u(x_2-M/2)\leq 4k_0 u(x_2)$, with \eqref{u(x_2)1}, one still has, for $y\leq r\leq q$,
\begin{equation}\label{u(x_2-y)small'}
u(x_2-y)\leq u(x_2-r)\leq 4k_0u(x_2)\leq \frac{8k_0f_0}{(1-\epsilon_0)\delta}\delta.
\end{equation}
Using the bound \eqref{f_0} on $f_0$ and \eqref{b1}, the above is at most
$$\frac{8k_0(\epsilon_1-\epsilon)^{1/\log b-1}}{(1-\epsilon_0)\kappa}\delta\leq \frac{1}{2c_1}\delta.$$
Thus, \eqref{u(x_2-y)small} holds. Repeating the argument below \eqref{u(x_2-y)small}, one gets Lemma \ref{conv_flat} in this case.

When $q=M/2$ but $u(x_2-M/2)>(4k_0)u(x_2)$, we truncate \eqref{truncation} before transforming this case to the previous case. Define
$$r'=\inf\{y\geq 0:u(x_2-y)>4k_0u(x_2)\}.$$
Then $0\leq r'<M/2$ and $u(x_2-r')\leq 4k_0u(x_2)$. By monotonicity of $u$, $u(x_2-y)\geq 4k_0u(x_1-y)$ for $y\in(r',r]$. Therefore, for $1\leq k\leq k_0$,
\begin{eqnarray*}
  &&\int_{r'}^{r}Q_{1,k}(u)(x_2-y)dg_k(y)
  -(1+\epsilon')\int_{r'}^{r}Q_{2,k}(u)(x_1-y)dg_k(y)\\
  &\geq & \int_{r'}^{r}u(x_2-y)dg_k(y)
  -2\int_{r'}^{r}k_0u(x_1-y)dg_k(y)\\
  &=& \int_{r'}^{r}\left(u(x_2-y)-2k_0u(x_1-y)\right)dg_k(y)\geq 0,
\end{eqnarray*}
which, together with \eqref{truncation}, yields the truncated inequality
$$\sum_{k=1}^{k_0}p_k\int_{-\infty}^{r'}Q_{1,k}(u)(x_2-y)dg_k(y)
  \leq(1+\epsilon')
 \sum_{k=1}^{k_0}p_k\int_{-\infty}^{r'}Q_{2,k}(u)(x_1-y)dg_k(y).$$
This is a analog of \eqref{truncation} with $r$ replaced by $r'$, and $u(x_2-r')\leq 4k_0u(x_2)$. Replacing $r$ by $r'$ in the argument starting from \eqref{u(x_2-y)small'}, one concludes the proof of Lemma \ref{conv_flat} in all cases.
\end{proof}

\begin{proof}[Proof of Lemma \ref{truncation_flat}]
\quad This lemma is to justify the flatness of the truncated integral. That is, we want to prove that mass from faraway does not affect the value of the integral in a significant way. This is almost guaranteed by the exponential decay of $g_{n,k}(\cdot)$. However, we need to control the difference between $Q_{1,k}(u)(x_2-y)$ and $Q_{2,k}(u)(x_1-y)$, using the lower bound \eqref{Q_1_LB1} on $Q_{1,k}(u)$ and the definition \eqref{Q_12} of $Q_{2,k}(u)$. Two different cases will be presented separately.

{\it Case (\romannumeral1):} when $r<y_0$, \eqref{r_steep} holds. Because of \eqref{Q_1Q_2} and $\epsilon<\epsilon'$, \eqref{truncation} will follow from
\begin{equation}\label{steepness}
  \int_{r}^{\infty}Q_{1,k}(u)(x_2-y)dg_k(y)-(1+\epsilon')
\int_{r}^{\infty}Q_{2,k}(u)(x_1-y)dg_k(y)\geq 0.
\end{equation}
To prove \eqref{steepness}, because of \eqref{Q_1_LB1}, it suffices to show that
$$\int_{r}^{\infty}u(x_2-y)dg_k(y)
-2k_0\int_{r}^{\infty}u(x_1-y)dg_k(y)\geq 0.\eqno{(\ref{steepness}')}$$
We break the left side into 3 pieces. First, by \eqref{r_steep},
\begin{eqnarray}\label{case2_1}
  &&\frac{1}{2}\int_{r}^{r+M/2}u(x_2-y)dg_k(y)
-2k_0\int_{r}^{r+M/2}u(x_1-y)dg_k(y)\nonumber\\
&=& \int_{r}^{r+M/2}\left(\frac{1}{2}u(x_2-y)
-2k_0u(x_1-y)\right)dg_k(y)\geq 0.
\end{eqnarray}
Second, because of assumption (MT2) (rapid decay of $\bar{g}_k(\cdot)$) and \eqref{M1}, one has
\begin{eqnarray}\label{case2_2}
  &&\frac{1}{2}\int_{r}^{r+M/2}u(x_2-y)dg_k(y)
-2k_0\int_{r+M/2}^{r+M}u(x_1-y)dg_k(y)\nonumber\\
&\geq & \frac{1}{4}u(x_2-r)\bar{g}_k(r)- 2k_0u(x_2-r)\bar{g}_k(r+M/2)\\
&\geq & (\frac{1}{4}-2k_0e^{-aM/2})u(x_2-r)\bar{g}_k(r)\geq 0.\nonumber
\end{eqnarray}
Third, again because of assumption (MT2) (rapid decay of $\bar{g}_k(\cdot)$) and \eqref{M1}, one has
\begin{eqnarray}\label{case2_3}
&&\int_{r+M/2}^{\infty}u(x_2-y)dg_k(y)
-2k_0\int_{r+M}^{\infty}u(x_1-y)dg_k(y)\nonumber\\
&\geq & \int_{r+M/2}^{\infty}u(x_2-(y-M/2))dg_k(y)-2k_0
\int_{r}^{\infty}u(x_2-y)dg_k(y+M)\nonumber\\
&=& \int_{r}^{\infty}u(x_2-y)dg_k(y+M/2)-2k_0
\int_{r}^{\infty}u(x_2-y)dg_k(y+M)\\
&\geq & (1-2k_0e^{-aM/2})\int_{r}^{\infty}u(x_2-y)dg_k(y+M/2)\geq 0\nonumber
\end{eqnarray}
Summing \eqref{case2_1}, \eqref{case2_2} and \eqref{case2_3}, one gets (\ref{steepness}'). Thus, \eqref{steepness} is verified in case (\romannumeral1), and \eqref{truncation} holds.

{\it Case (\romannumeral2):} when $r=y_0$, \eqref{r_steep} may not be true. However, the difference between the two sides of \eqref{truncation} is
\begin{eqnarray*}
  &&\sum_{k=1}^{k_0}p_k\int_{-\infty}^{r}Q_{1,k}(u)(x_2-y)dg_k(y)
  -(1+\epsilon')\sum_{k=0}^{k_0}p_k\int_{-\infty}^{r}Q_{2,k}(u)(x_1-y)dg_k(y)\nonumber\\
 &\leq & \sum_{k=1}^{k_0}p_kg_k*Q_{1,k}(u)(x_2)
  -(1+\epsilon')\sum_{k=1}^{k_0}p_k\int_{-\infty}^{r}Q_{2,k}(u)(x_1-y)dg_k(y).
\end{eqnarray*}
Recall that $\epsilon'=\epsilon+\delta$. \eqref{Q_1Q_2} implies that the above quantity is less than or equal to
\begin{eqnarray*}
  &&(1+\epsilon)\sum_{k=1}^{k_0}p_kg_k*Q_{2,k}(u)(x_1)
  -(1+\epsilon')\sum_{k=1}^{k_0}p_k\int_{-\infty}^{r}Q_{2,k}(u)(x_1-y)dg_k(y)\\
  &=&(1+\epsilon')\sum_{k=1}^{k_0}p_k\int_{r}^{\infty}Q_{2,k}(u)(x_1-y)dg_k(y)
  -\delta\sum_{k=1}^{k_0}p_kg_k*Q_{2,k}(u)(x_1).
\end{eqnarray*}
Since $r=y_0=\frac{1}{a}\log\frac{2k_0}{\delta f_0}$, the assumption (MT2) implies $\bar{g}_k(r)\leq e^{-ay_0}=\frac{\delta f_0}{2k_0}$. $Q_{2,k}(u)\leq k_0$, \eqref{bound'} and \eqref{f_0} yield that the above quantity again does not exceed
\begin{eqnarray*}
  (1+\epsilon')k_0 \frac{\delta f_0}{2k_0}-\delta f_0\leq 0.
\end{eqnarray*}
So \eqref{truncation} is proved in case (\romannumeral2). This completes the proof of the lemma.
\end{proof}

\section{Tightness for identical marginals}\label{assumption n result 2}

In this section, we discuss the tightness problem in the case when all the marginal distributions at the same level are the same, i.e., $g_{n,k}(\cdot)=g_n(\cdot)$ does not depend on the number of offsprings. Compared with the assumptions made in Section \ref{assumption n result}, we relax the bounded support assumption (B1) on $p_{n,k}$s, at the price of a uniform marginal assumption on $G_{n,k}$ (see (MT0') below). Namely, we assume
\begin{itemize}
  \item[(B1')] There exist positive real numbers $m_0$ and $m_1$ such that $\inf_n\{\sum_{k=1}^{\infty}kp_{n,k}\}>m_0>1$ and $\sup_n\sum_{k=1}^{\infty}k^2p_{n,k}<m_1$.
  \item[(MT0')] $g_{n,k}(\cdot)=g_n(\cdot)$ for all $k\geq 1$.
  \item[(MT1')] For some fixed $\epsilon_0< \frac{1}{4}\log m_0\wedge 1$, there exists an $x_0$ such that $\bar{g}_n(x_0)\geq 1-\epsilon_0$ for all $n$, where $\bar{g}_n(x)=1-g_n(x)$. By shifting, we will assume that $x_0=0$, that is, $\bar{g}_n(0)\geq 1-\epsilon_0$.
  \item[(MT2')] There exist $a>0$ and $M_0>0$ such that $\bar{g}_n(x+M)\leq e^{-aM}\bar{g}_n(x)$ for all $n$ and $M>M_0$, $x\geq 0$.
  \item[(GT')] For any $\eta_1>0$, there exists a $B>0$ such that $G_{n,k}(B,\dots,B)\geq 1-\eta_1$ and $\bar{g}_n(-B)\geq 1-\eta_1$ for all $n$ and $k$.
\end{itemize}
Then we still have the following tightness result.
\begin{theorem}\label{tightness 2}
  Under assumptions (B1'), (MT0'), (MT1'), (MT2') and (GT'), the family of the recentered maxima distribution $\{F_n\left(\cdot-Med(F_n)\right)\}$ is tight.
\end{theorem}
Since the proof is similar to the proof of Theorem \ref{tightness}, we only bring a sketch. The argument is based on the following recursion inequality, another form of \eqref{bound} under the assumption (MT0'),
\begin{equation}\label{g_bounds}
  g_{m}*\left(\sum_{k=1}^{\infty}p_{m,k}Q_{1,k}(\bar{F}_n^{m+1})\right)(x)\leq \bar{F}_n^m(x)\leq g_m*\left(\sum_{k=1}^{\infty}p_{m,k}Q_{2,k}(\bar{F}_n^{m+1})\right)(x),
\end{equation}
where $Q_{1,k}$ and $Q_{2,k}$ are defined as \eqref{Q_12}. Set
\begin{equation}\label{Q_1'}
  Q_{m,(1)}(u)=\sum_{k=1}^{\infty}p_{m,k}Q_{1,k}(u),
\end{equation}
and
\begin{equation}\label{Q_2'}
  Q_{m,(2)}(u)=\sum_{k=1}^{\infty}p_{m,k}Q_{2,k}(u).
\end{equation}
Although the difference between $Q_{1,k}$ and $Q_{2,k}$ gets bigger as $k$ becomes bigger, the weighted function $Q_{m,(1)}$ and $Q_{m,(2)}$ still behave nicely and possess an analog of Lemma \ref{Q_1_LB12}.
\begin{lemma}\label{Q_1'_LB12}
Let $Q_{m,(1)}$ and $Q_{m,(2)}$ be defined as in \eqref{Q_1'} and \eqref{Q_2'} respectively, then it follows from assumption (B') that
\begin{equation}\label{Q_1'_LB1}
  Q_{m,(1)}(u)>u,
\end{equation}
and
\begin{equation}\label{Q_1'_LB2}
  Q_{m,(2)}(u) - c_2 u^2 \leq Q_{m,(1)}(u)\leq Q_{m,(2)}(u)\leq \sqrt{m_1} u.
\end{equation}
\end{lemma}

Lemma \ref{truncation_flat} relies on the facts that $Q_{1,k}(u)\geq u$ and $Q_{2,k}(u)\leq k_0u$, and Lemma \ref{conv_flat} relies on the fact that $Q_{1,k}(u)\geq Q_{2,k}(u)-c_1u^2$. Therefore, with a modification of $q$ and $r$, we can prove analogs of those two lemmas due to the bounds in Lemma \ref{Q_1'_LB12}. An analog of Proposition \ref{prop_fact} then follows. Proposition \ref{Lyapunov} and Corollary \ref{RT} hold under the new assumptions in this section.

%Equipped with the bounds in Lemma \ref{Q_1'_LB12}, one can easily prove an analog of Proposition \ref{prop_fact}. Thus the conclusion of Proposition \ref{Lyapunov} and Corollary \ref{RT} hold under the new assumptions in this section.

Assumption (GT') plays a role as (GT) in connecting the left and right tail behavior. Specifically, it guarantees Lemma \ref{Assumption in BZ}, Lemma \ref{LnR} and Proposition \ref{tightprop} under the new settings. Theorem \ref{tightness 2} follows immediately as Theorem \ref{tightness}.

%\section{Open Question}

\section*{Appendix}

\begin{proof}[Proof of Lemma 7 assuming Lemma 6]
This will be proved by contradiction. Assume that neither (a) nor (b) in lemma 7 holds, i.e.,
$$u(x_2-y)>(1+\epsilon^{(3)})u(x_1-y)\;\;\mbox{for all}\;\;y\leq r'\wedge M,\eqno{(\bar{a})}$$
and
$$u(x_2-y)>(1+\epsilon''-\delta e^{ay/8})u(x_1-y)\;\;\text{for all}\;\;y\in(M,r'].\eqno{(\bar{b})}$$
If $r'\leq M$, then only ($\bar{a}$) holds and it implies that
$$\int_{-\infty}^{r'}u(x_2-y)dg_k(y)>(1+\epsilon^{(3)})
\int_{-\infty}^{r'}u(x_1-y)dg_k(y).$$
Since $\epsilon^{(3)}=\epsilon''+\delta>\epsilon''$, this is a contradiction to \eqref{truncation}. So we are done.
If $r'>M$, then $r'=r$ and ($\bar{a}$) and ($\bar{b}$) imply that
$$\int_{-\infty}^{M}u(x_2-y)dg_k(y)>(1+\epsilon^{(3)})
\int_{-\infty}^{M}u(x_1-y)dg_k(y),$$
and
$$\int_{M}^{r}u(x_2-y)dg_k(y)>
\int_{M}^{r}(1+\epsilon''-\delta e^{ay/8})u(x_1-y)dg_k(y).$$
Summing the above two inequality, one gets that
\begin{eqnarray*}
  &&\int_{-\infty}^{r}u(x_2-y)dg_k(y)>(1+\epsilon'')
\int_{-\infty}^{r}u(x_1-y)dg_k(y)\\
&&\;\;\;\;\;\;\;\;\;\;\;+\delta \left[ \int_{-\infty}^{M}u(x_1-y)dg_k(y)- \int_{M}^{r} e^{ay/8}u(x_1-y)dg_k(y)\right].
\end{eqnarray*}
We claim that
\begin{equation}\label{claim}
  \int_{-\infty}^{M}u(x_1-y)dg_k(y)- \int_{M}^{r} e^{ay/8}u(x_1-y)dg_k(y)\geq 0,
\end{equation}
which will imply a contradiction of \eqref{truncation} and complete the proof. It thus remains to prove the claim \eqref{claim}. Since $q\geq r>M/2$, one has $u(x_2-y)\leq (4k_0)^2u(x_1-y)$ for all $y\in[M/2,r]$. By \eqref{M1}, one can bound the second integral in the left side of the above inequality as follows.
\begin{eqnarray*}
  &&\int_{M}^{r} e^{ay/8}u(x_1-y)dg_k(y) =  \sum_{l=1}^{\infty}\int_{lM}^{lM+M} e^{ay/8}u(x_1-y)1_{\{y\leq r\}}dg_k(y)\\
  &\leq & \sum_{l=1}^{\infty}\int_{lM}^{lM+M} e^{alM/8+aM/8}(4k_0)^{2l+2}u(x_1-M/2)dg_k(y)\\
  &\leq & \sum_{l=1}^{\infty} e^{alM/8+aM/8}(4k_0)^{2l+2}u(x_1-M/2)
  \bar{g}_k(lM)\\
  &\leq & \sum_{l=1}^{\infty} e^{alM/8+aM/8}(4k_0)^{2l+2}u(x_1-M/2)
  e^{-alM+aM/2}\bar{g}_k(M/2)\leq \frac{1}{4}u(x_1-M/2)\bar{g}_k(M/2).\\
\end{eqnarray*}
But the last term does not exceed
$$\int_{M/2}^Mu(x_1-y)dg_k(y)\leq \int_{-\infty}^M u(x_1-y)dg_k(y).$$
So the proof of \eqref{claim} is complete. We are done with proving Lemma \eqref{flatness}.
\end{proof}

\section*{Acknowledgement}

I would like to thank my advisor, Ofer Zeitouni, for posing this problem and guiding me through this work.


\begin{thebibliography}{8}
  \bibitem{AR} L. Addario-Berry and B.A. Reed, \emph{Minima in
 branching random walks}, Annals of probability 37 (2009), pp. 1044--1079.
  \bibitem{Br1} M. Bramson, \emph{Maximal displacement of branching
 Brownian motion}, Comm. Pure Appl. Math. 31 (1978), pp. 531-581.
  \bibitem{BZ} M. Bramson and O. Zeitouni, \emph{Tightness for a family of recursion
 equations}, Annals of probability, Vol 37 No. 2 (2009), pp. 615-653.
  \bibitem{BZ3} M. Bramson and O. Zeitouni, \emph{Tightness of the recentered maximum of the two-dimensional discrete Gaussian free field}, Arvix:1009.3443v1 (2010).
  \bibitem{DH} F.M. Dekking and B. Host, \emph{Limit distributions for minimal displacement of branching random
 walks}, Probab. Theory Relat. Fields 90 (1991), pp. 403-426.
\end{thebibliography}
\end{document}